\documentclass[11pt,twoside,a4paper]{article}

\usepackage{amssymb}
\usepackage{amsmath}
\usepackage{amsthm}
\usepackage{color}
\usepackage{mathrsfs}

\usepackage{amsfonts}
\usepackage{amssymb}
\usepackage{amsmath}
\usepackage{amsthm}
\usepackage{bbm}
\usepackage{verbatim}
\usepackage{url}

\allowdisplaybreaks

\newcommand{\eps}[0]{\varepsilon}

\def\supp{{\mathop\mathrm{\,supp\,}}}
\def\dist{{\mathop\mathrm{\,dist\,}}}

\newtheorem{thm}{Theorem}[section]
\newtheorem{lem}[thm]{Lemma}
\newtheorem{prop}[thm]{Proposition}
\newtheorem{clm}[thm]{Claim}

\newcommand{\R}{\mathbb R}

\numberwithin{equation}{section}

\begin{document}
\arraycolsep=1pt

\title{\Large\bf Even singular integral operators that are well behaved on a purely unrectifiable set }
\author{Benjamin Jaye and Manasa N. Vempati}

\date{}
\maketitle

\begin{center}
\begin{minipage}{13.5cm}\small

{\noindent  {\bf Abstract:} We prove the existence of a $(d-2)$-dimensional purely unrectifiable set upon which a family of \emph{even} singular integral operators is bounded.
}

\end{minipage}
\end{center}

\section{Introduction}

Understanding the geometry of a measure $\mu$ for which an associated \emph{odd} singular integral operator is bounded in $L^2(\mu)$ is a heavily studied problem in harmonic analysis. For instance, the question of whether the $L^2(\mu)$ boundedness for $s$-dimensional Riesz transforms implies the rectifiability (and closely related conditions) of the measure $\mu$, is known as the David-Semmes question \cite{DS}. This is only known when $s=1$ \cite{mmv, dav1, ler} and $s=d-1$ \cite{env2, ntov, dt, to}. There are several positive and negative results that have been proved for a wide variety of odd kernels, see e.g. \cite{Ch, cpmt, hu, JN, MP}.

Taking our inspiration from \cite{MOV, MOV2}, the goal of this note is to take a first step in the study of the problem for \emph{even kernels} by characterizing the even kernels that can be bounded in $L^2$ on a particular class of purely unrectifiable sets, namely, the natural analogues of the sets first considered in $\R^2$ in \cite{JN} and recently generalized by \cite{MP} to co-dimension one sets in $\R^d$.



Let us recall that a set $A\subset \mathbb{R}^d$ is called \emph{$m$-rectifiable} if there are Lipschitz maps $f_i: \R^m \rightarrow \R^d $ for all $i = 1,2,...,$
such that

\begin{equation*}
    \mathscr{H}^m(A\setminus\cup f_i(\R^m)) =0.
\end{equation*}

(Here $\mathscr{H}^m$ denotes the $m$-dimensional Hausdorff measure.)  In contrast, a set $B$ is \emph{$m$-purely unrectifiable} if $ \mathscr{H}^m(B\cap A) =0$ for every $m$-rectifiable set $A \subset \R^d$. 

A measure $\mu$ is said to have \emph{$m$-growth} if there exists a constant $C>0$, such that $\mu(B(x,R))\leq Cr^{m}$ for every ball $B(x,r) \subset \mathbb{R}^d$. Finally, we say $\mu$ is an \emph{$m$-dimensional measure} if $\mathscr{H}^{m}(\supp{\mu})< \infty$ and $\mu$ has $m$-growth.

Throughout the paper, we will denote by $\Omega$ an even H\"{o}lder continuous function on $\mathbb{S}^{d-1}$ with $\int_{\mathbb{S}^{d-1}}\Omega(\xi)d\mathscr{H}^{d-1}(\xi)=0.$ For a finite measure $\mu$, the $m$-dimensional SIO associated to $\Omega$ is bounded in $L^2(\mu)$ if there is a constant $C>0$ such that for every $f\in L^2(\mu)$
$$\sup_{\varepsilon>0}\int_{\mathbb{R}^d}\Bigl|\int_{y:|x-y|>\varepsilon}\frac{\Omega\bigl(\frac{x-y}{|x-y|}\bigl)}{|x-y|^{m}}f(y)d\mu(y)\Bigl|^2d\mu(x)\leq C\|f\|^2_{L^2(\mu)}.
$$

Employing the $T(1)$-theorem for spaces of non-homogeneous type \cite{ntv}, provided that $\mu$ has $m$-growth one can reduce the study of operator boundedness to understanding whether the potential, defined for $x\notin \supp(\mu)$ by $$
T(\mu)(x) = \int_{\R^d}\frac{\Omega\bigl(\frac{x-y}{|x-y|}\bigl)}{|x-y|^{m}}d\mu(y),$$
belongs to $L^{\infty}(\R^d\backslash \supp(\mu)).$

Now let us state the main result of this paper, which concerns even SIOs of co-dimension $2$.

\begin{thm}\label{maintheo}
Let $d\geq 3$. There exists a $(d-2)$-purely unrectifiable set $E$ and a $(d-2)$-dimensional probability measure supported on $E$ such that the $(d-2)$-dimensional potential associated to $\Omega$ belongs to $L^{\infty}(\R^d\backslash\supp(\mu))$ if and only if
\begin{equation}\label{moment}\int_{\mathbb{S}^{d-1}}\xi_i\cdot\xi_j\cdot\Omega (\xi)\,d\mathscr{H}^{d-1}(\xi)= 0\text{ for all }i,j\in \{1,\dots, d\}.\end{equation}
Additionally, unless $\Omega$ vanishes identically, the SIO associated to $\Omega$ fails to exist in the sense of principal value $\mu$-almost everywhere.
\end{thm}


It is interesting to note that for the singular integral operator  associated to the measure $\mu$ we construct in Theorem \ref{maintheo}, the two properties of $L^2(\mu)$-boundedness and existence in the principle value are quite distinct.

The class of kernels that satisfy the hypothesis for our main result above is non-empty.  To construct an example, fix a non-negative function $\varphi\in C^1([0,1])$ with $\varphi(1)=0$. For $\xi = (\xi_1,\xi_2,\xi')\in \mathbb{R}^d$ first define
$$a(\xi_1,\xi_2,\xi') = (\xi_2^2-\xi_1^2)\varphi(\xi_1^2+|\xi'|^2)\varphi(\xi_2^2+|\xi'|^2).
$$
For any $\xi'\in \mathbb{R}^{d-2}$ with $|\xi'|\leq 1$, the integral of $a(\cdot,\cdot, \xi')$ over the quarter circle $\{(\xi_1,\xi_2):\; \xi_1\geq0,\,\xi_2\geq 0,\, |\xi_1|^2+|\xi_2|^2=1-|\xi'|^2\}$ equals $0$, and $a$ vanishes if either $|\xi_1|^2$ or $|\xi_2|^2$ equals $1-|\xi'|^2$
For $\xi = (\xi_1,\xi_2,\xi')\in \mathbb{S}^{d-1}$ set
\begin{align*}
&\Omega(\xi_1,\xi_2,\xi') = a(\xi_1,\xi_2,\xi') \quad \textit{if} \quad \{\xi_1,\xi_2\geq 0 \} \cup \{\xi_1,\xi_2\leq 0 \}\\
&\Omega(\xi_1,\xi_2,\xi') = a(\xi_2,\xi_1,\xi') = -a(\xi_1,\xi_2,\xi')  \quad \textit{if} \quad \{\xi_1\geq 0 ,\xi_2<0\} \cup \{\xi_1\leq 0,\xi_2\geq 0 \}.
\end{align*}
The function $\Omega$ is a H\"{o}lder continuous mean-zero even function on the sphere that satisfies (\ref{moment}). 


\subsection{Acknowledgements}  The authors were supported by the NSF through grants DMS--2049477 and DMS--2103534.  This work was completed while the authors were in residence at ICERM during the semester program on Harmonic Analysis and Convexity.

\section{The reflectionless property}

We will denote by $m_d$ the $d$-dimensional Lebesgue measure, and we often denote the surface area measure on $\mathbb{S}^{d-1}$ by $\sigma$.  The next lemma is the key to our construction, and is based around the proof of Lemma 3 of \cite{MOV2}.

\begin{lem}[The Reflectionless Property]\label{reflectionless}
Let $x_0\in \mathbb{R}^d$, $r>0$. The condition (\ref{moment}) holds if and only if 
\begin{equation}\label{refless}
 \int_{B(x_0,r)}K(x-y) dm_d(y) =0  \text{ for any }x\in B(x_0,r).
\end{equation}
\end{lem}

\begin{proof}
Without loss of generality we may assume that $x_0=0$ and $r=1$.  Employing the mean zero property of $\Omega$, first observe that
$$ \int_{B(0,1)}K(x-y) dm_d(y) = \int_{B(x,1+|x|)\backslash B(0,1)}K(x-y)dm_d(y).
$$
Express this integral in terms of the polar coordinates $y=x+r\zeta$ centered at $x$. Setting $r(x,\xi)$ to be the (smallest) solution of $|x+r(x,\xi)\cdot\xi|=1$, we get
\begin{align*}
 \int_{B(0,1)}K(x-y) dm_d(y)
 & = \int_{|\zeta|=1}\int_{r(x,\zeta)}^{1+|x|}{1\over r^{d-2}}\Omega(\zeta) r^{d-1} dr d\sigma(\xi)\\
&= {1\over 2}\int_{|\zeta|=1} (1+|x|)^2 \Omega(\zeta)  d\sigma(\zeta) - {1\over 2}\int_{|\zeta|=1} (r(x,\zeta))^2 \Omega(\zeta)  d\sigma(\zeta) \\
&=: I - II
 \end{align*}

The term $I$ vanishes as $\int_{\mathcal{S}^{d-1}}\Omega(x)d\sigma(x) =0$.

Set $U^+$ be the half of the unit sphere above the hyperplane $\{x_d = 0\}$. Since $\Omega$ is an even function,

\begin{equation*}
II = {1\over 2}\int_{|\zeta|=1} (r(x,\zeta))^2 \Omega(\zeta)  d\sigma(\zeta)  = {1\over 2}\int_{U^+} \bigl[r(x,\zeta)^2 + r(x,-\zeta)^2\bigl]  \Omega(\zeta)  d\sigma(\zeta) \end{equation*}

A simple computation yields $(r(x,\zeta)^2 + r(x,-\zeta)^2) = 4(x\cdot\zeta)^2 - 2(|x|^2-1)$, which leads to
\begin{align*}
II & = {1\over 2}\int_{U^+} (r(x,\zeta)^2 + r(x,-\zeta)^2)  \Omega(\zeta)  d\sigma(\zeta)   \\&=   {1\over 2}\int_{U^+}  4(x\cdot \zeta)^2 \Omega(\zeta)  d\sigma(\zeta)  - {1\over 2}\int_{U^+} 2(|x|^2-1) \Omega(\zeta)  d\sigma(\zeta) =: III + IV.
\end{align*}

Since $\Omega$ has mean-zero over $U_+$, the term $IV$ vanishes. 

We conclude that (\ref{refless}) holds if and only if the term $III$ vanishes for every $x\in B(0,1)$, which is in turn equivalent to the condition (\ref{moment}).  This completes the proof of the lemma.
\end{proof}

\section{ Construction of the zero lower density set and the associated measure}

The construction our zero lower density set follows along the same lines as the papers \cite{JN} and \cite{MP}.

\subsection{The set $E$}  Set $\kappa_d = {\pi^{d\over 2}\over \Gamma({d\over 2}+1)}$ to be the volume of the $d$-dimensional unit ball.  

\begin{lem}\label{packingballs}
One can pack $(\frac{R}{r})^{d-2}$ pairwise essentially disjoint cubes of side length $\sqrt[d]{\kappa_d r^{d-2}R^2}$ into a ball of radius $R\bigg(1+\sqrt{d}\sqrt[d]{\kappa_d}\sqrt[d]{r^{d-2}\over R^{d-2}}\bigg)$.
\end{lem}

\begin{proof}
Without loss of generality we can assume that our ball is centered at the origin. Now we will consider a cubic grid of mesh size $\sqrt[d]{\kappa_d r^{d-2}R^2}$. Suppose now that the cubes $Q_1,Q_2,....,Q_M$ intersect $B(0,R)$. These cubes are contained in the ball centred at $0$ with radius $R\bigl(1+\sqrt{d}\sqrt[d]{\kappa_d}\sqrt[d]{r^{d-2}\over R^{d-2}}\bigl)$.  Finally, since
\begin{equation*}
    M\kappa_d r^{d-2}R^2= \sum_{j=1}^{M}m_d(Q_j)>m_d(B(0,R)) = \kappa_d R^d,
\end{equation*}
 we have $M>{R^{d-2}\over r^{d-2}}$, and the lemma follows.
\end{proof}


 To begin let us consider a sequence $\{r_k\}_{k\geq 0}$ that tends to zero quickly and such that $r_0=1$ and $r_{k+1}< {r_k \over B}$ for an absolute constant $B$ which will be chosen later. Additionally shall assume that ${r_k \over r_{k+1}}\in \mathbb{N}$ and ${1\over r_k}\in \mathbb{N}$.

Set $\widetilde{B}_{1}^{0} = B(0,1)$. We will construct the set iteratively. Given the $k$-th generation of ${1\over r_k^{d-2}}$ balls $\Tilde{B}_{j}^{k}$ of radius $r_k$, we proceed to the $(k+1)$-st generation as follows: for each ball $\widetilde{B}_{j}^{k}$ we apply Lemma \ref{packingballs} with $R= r_k$ and $r= r_{k+1}$, so we find $({r_k\over r_{k+1}})^{d-2}$ of pairwise disjoint cubes $Q_{l}^{k+1}$ balls of sidelength $\sqrt[d]{\kappa_d r_{k+1}^{d-2}r_k^2}$ contained in ball $(1+A\sqrt[d]{r^{d-2}\over R^{d-2}})\widetilde{B}_{j}^{k}$ where $A=\sqrt{d}\sqrt[d]{\kappa_d}$.

Set $\widetilde{B}_{j}^{k+1}= B(z_{l}^{k+1}, r_{k+1})$ where $z_{l}^{k+1}$ denotes the center of the cube $Q_{l}^{k+1}$. We carry out this process for each ball $\widetilde{B}_{j}^{k}$ from the $k$-th generation. In total, we get ${1\over r_{k+1}^{d-2}}$ balls $\widetilde{B}_{j}^{k+1}$ in the $(k+1)$-st level.

Set $\delta_{k+1} = A\sqrt[d]{r_{k+1}^{d-2}\over r_{k}^{d-2}}$ and
\begin{equation*}
  B_{j}^{k}  = (1+\delta_{k+1})\widetilde{B}_{j}^{k} \quad \textit{and} \quad E^{k} = \bigcup_{j\geq 1}B_{j}^{k}.
\end{equation*}

We will frequently make use of the following properties of the construction

\begin{itemize}
    \item[(i)] For each $k\geq 1$,  $$\bigcup_{l}Q_{l}^{k+1} \subset E^{k}.$$
    \item[(ii)] For each $k\geq 1$, we have $B_{j}^{k} \subset Q_{j}^{k}$, and moreover (provided that $B$ is chosen appropriately)
     $$\text{dist}(B_{j}^{k},\partial Q_{j}^{k}) \geq \frac{1}{4}\sqrt[d]{\kappa_d r_{k}^{d-2}r_{k-1}^2}.$$
    
    
    \item[(iii)] For each $k\geq 1$ and for $i\neq j$, $\text{dist}(B_{j}^{k}, B_{i}^{k}) \geq \frac{1}{4}{\sqrt[d]{\kappa_dr_{k}^{d-2}r_{k-1}^2}}.$
\end{itemize}
Observe that for each $k\geq 0$ we have $E^{k+1}\subset E^{k}$ and now set $E= \bigcap_{k\geq 0}E^k$.  It is not hard to check that the set $E$ satisfies $0<\mathscr{H}^{d-2}(E)<\infty$, and $\liminf_{r\to 0}\frac{\mathscr{H}^{d-2}(E\cap B(x,r))}{r^{d-2}}=0$ for every $x\in E$.  Consequently the set $E$ is $(d-2)$-purely unrectifiable \cite{Mat}.

It will be convenient to use the following notation: Each $x\in E^k$ is contained in a some unique ball $B_j^k$ and in a unique cube $Q_j^k$, we will denote these by $B^k(x)$ and $Q^k(x)$ respectively.

\subsection{The measure $\mu$}  Set 

\begin{equation*}
   \mu_{j}^{k} = {1\over r_k^2}\chi_{\Tilde{B}_{j}^{k}}m_d \quad \textit{and} \quad \mu^k = \sum_{j}\mu_{j}^{k}.
\end{equation*}

Observe that $\supp(\mu^k) \subset E^k$ and $\mu^k(\mathbb{R}^d) = 1$ for all $k$.  The following properties hold for the measures $\mu^{k}$:

\begin{itemize}
    \item [(a)] $\supp(\mu^k) \subset  \bigcup_{j\geq 1}B_{j}^{m}$ if $k\geq m$.
    \item[(b)] $\mu^k(B_{j}^{m}) = r_m^{d-2}$ for $k\geq m$.
    \item[(c)] There exists a constant $C>0$ such that for any $k$ and ball $B(x,r)$, $$\mu^k(B(x,r))\leq Cr^{d-2}.$$
\end{itemize}

Properties (a) and (b) follow immediately from construction.  For (c), first note that for $r\geq 1$, this property is clear as $\mu^k$ is a probability measure. If $0<r<1$, then $r\in (r_{m+1},r_{m})$ for some $m \in \mathbb{N}$. In the case when $m\geq k$, $r_m \leq r_k$, and hence the ball $B(z,r)$ intersects with at most one $B_j^k$. Therefore

\begin{equation*}
 \mu^k(B(z,r)) = {1\over r_k^2}m_d(  B(z,r) \cap \widetilde{B}_j^k) \leq {r^d\over r_k^2} \leq r^{d-2}.
\end{equation*}

On the other hand, if $k\geq m$, then property (iii) of the construction ensures that $B(z,r)$ intersects at most $1+ C{r^d\over r_{m+1}^{d-2}r_{m}^2}$ balls $B_j^{m+1}$. Property $(b)$, then ensures that
\begin{equation*}
 \mu^k(B(z,r)) = \sum_{j} \mu^k(B(z,r)\cap B_{j}^{m+1}) \leq \bigg(1+ C{r^d\over r_{m+1}^{d-2}r_{m}^2}\bigg)r_{m+1}^{d-2} \leq Cr^{d-2},
\end{equation*}
and (c) is proved.

Finally, passing to a subsequence if necessary, the measures $\mu^k$ converge weakly to a $(d-2)$-dimensional measure $\mu$ supported on $E$.

\section{The boundedness of the potential associated to $\Omega$}

Fix $\alpha\in (0,1]$ to be the H\"{o}lder exponent of $\Omega$.  All absolute constants in this section may depend on dimension, and the quantity $$\|\Omega\|_{C^{\alpha}(\mathbb{S}^{d-1})}: = \sup_{\omega,\xi\in \mathbb{S}^{d-1}}\frac{|\Omega(\omega)-\Omega(\xi)|}{|\omega-\xi|^{\alpha}}$$ without further mention.  We will write $P\lesssim Q$ to mean that $P\leq CQ$ for an absolute constant $C>0$.

 We shall henceforth assume that $E$ is constructed so that \begin{equation}\label{deltasum}\sum_{k\geq 1}\delta_k^{2\alpha/d} < \infty.\end{equation}

Our first goal will be to show that the property (\ref{moment}) will ensure that $\|T_{\mu}(1)\|_{L^{\infty}(\mathbb{R}^d\setminus \supp{\mu})}<\infty$. This will in turn follow from the weak convergence of $\mu^k$ to $\mu$ and the following proposition. 

\begin{prop}\label{mainprop}
There is a constant $C>0$ such that the following holds: Provided $\dist(x,\supp{(\mu)}) = \epsilon>0$, then for any $m\in \mathbb{N}$ with $r_m < {\epsilon \over 2d}$,
\begin{equation*}
    \bigg|\int K(x-\zeta)d\mu^m(\zeta)\bigg|\leq C.
\end{equation*}
\end{prop}

\begin{proof}To begin the proof of Proposition \ref{mainprop}, fix $x^{*} \in \supp{(\mu)}$ with $\dist{(x,x^*)} = \epsilon$. Select $m$ satisfying $r_m < {\epsilon \over 2d}$, and let $n$ be the least integer such that $r_n \leq \epsilon$ (hence $m\geq n)$. Observe that

\begin{align*}
   \int K(x-\zeta)d\mu^m(\zeta) & = \int_{B^n(x^*)} K(x-\zeta)d\mu^m(\zeta) + \sum_{k=1}^{n}\int_{B^{k-1}(x^*)\setminus B^{k}(x^*)} K(x-\zeta)d\mu^m(\zeta)\\
   &= A_1 + A_2.
\end{align*}

To estimate $A_1$ observe that every $\zeta \in \supp(\mu^m)$ is contained in a ball $B_j^m$ of radius $(1+\delta_{m+1})r_m$.  Therefore

\begin{equation}\label{absolutevalue}
    \dist(x, \supp{(\mu^m)}) \geq \epsilon - (1+\delta_{m+1})r_m \geq {\epsilon \over 2}.
\end{equation}

Consequently, using property $(b)$ of the measure $\mu^m$, 

\begin{equation*}
    |A_1| \leq  \int_{B^n(x^*)} |K(x-\zeta)|d\mu^m(\zeta) \lesssim \frac{\mu^m(B^n(x^*))}{\epsilon^{d-2}} \lesssim \frac{r_n^{d-2}}{\epsilon^{d-2}} \lesssim 1.
\end{equation*}

To estimate the term $A_2$ we make the following claim: For some constant $C>0$ such that for any $k \in \{1,2...,n\},$

\begin{equation}\label{ineqsecondterm}
  \bigg|\int_{B^{k-1}(x^*)\setminus B^{k}(x^*)} K(x-\zeta)d\mu^m(\zeta) \bigg| \lesssim \delta_k^{\alpha 2/d} + \sqrt[d]{\frac{\epsilon^2}{r_{k-1}^2}}.
\end{equation}

Employing \eqref{ineqsecondterm} yields (recalling the assumption (\ref{deltasum}))

\begin{equation*}
    |A_2| \lesssim 
 \sum_{k=1}^{n}\Bigl(\delta_k^{2\alpha/d}+\sqrt[d]{\epsilon^2 \over r_{k-1}^2}\Bigl) \lesssim 1,
\end{equation*}
and Proposition \ref{mainprop} follows.  Therefore our goal will be to prove our claim \eqref{ineqsecondterm}.  To do so we will appeal to the following comparison lemma.

\begin{lem}\label{kernelestimates}
Let $x_0 \in \mathbb{R}^d$. Fix $r,R \in (0,1]$ with $r$  smaller than $R.$ Let $Q\in \mathbb{R}^d$ be some cube centered at $x_0$ with the sidelength $l(Q) = \sqrt[d]{\kappa_d r^{d-2}R^2}$ and let $B = B(x_0,2r).$
Suppose that $\nu_1, \nu_2$ are Borel measures with $\supp{(\nu_1)}\subset Q$,
$\supp{(\nu_2)} \subset B$ and $\nu_1(\mathbb{R}^d) = \nu_2(\mathbb{R}^d).$ Then for any $x\in \mathbb{R}^d$ with $\dist{(x, Q)} \geq {\sqrt[d]{\kappa_d r^{d-2}R^2} \over 8},$ we have

\begin{equation}\label{kernelsmoothness}
    \bigg|\int_{Q} K(x-\zeta)d\nu_1(\zeta) - \int_{B} K(x-\zeta)d\nu_2(\zeta)\bigg| \lesssim  
   \int_{Q}{(r^{d-2}R^2)^{\alpha/d}d\nu_1(\zeta)\over |x-\zeta|^{d-2+\alpha}} + \int_{B} {r^{\alpha}d\nu_2(\zeta)\over |x-\zeta|^{d-2+\alpha}}.
\end{equation}

\end{lem}

\begin{proof} We can set $x_0=0$, without loss of generality. For any $\zeta \in Q$ and for any $x$ such that $\dist{(x,Q)} \geq {\sqrt[d]{\kappa_d r^{d-2}R^2} \over 8}$, we have $|x|\approx |x-t\xi|$ for any $t\in [0,1]$, and so we have the standard kernel estimate

\begin{equation}\label{kernelsize}
 |K(x-\zeta) - K(x)| \lesssim  {|\zeta|^{\alpha}\over |x-\zeta|^{d-2+\alpha} },
\end{equation}
Therefore
\begin{align*}
&\bigg|\int_{Q} K(x-\zeta)d\nu_1(\zeta)  - \int_{B} K(x-\zeta)d\nu_2(\zeta)\bigg| \\
&\leq  \Bigl|\int_{Q} [K(x-\zeta) - K(x)]d\nu_1(\xi)\Bigl| +\Bigl|\int_{B} [K(x-\zeta)-K(x)]d\nu_2(\zeta)\Bigl| \\
&\lesssim  \int_{Q}{|\zeta|^{\alpha} \over |x-\zeta|^{d-2+\alpha}}d\nu_1(\zeta) + \int_{B}{|\zeta|^{\alpha} \over |x-\zeta|^{d-2+\alpha}}d\nu_2(\zeta).
\end{align*}
Observing that $|\zeta|\lesssim \sqrt[d]{r^{d-2}R^2}$ for $\xi\in Q$, while $|\zeta|\lesssim r$ for $\xi\in B$, completes the proof of lemma.
\end{proof}

Now we will proceed to prove the claim \eqref{ineqsecondterm}. Denote by $\mathcal{S}$ the collection
\begin{equation*}
  \mathcal{S} = \{j: B_j^k \neq B^k(x^*)\text{ and }B_j^k \subset B^{k-1}(x^*)\}.
\end{equation*}

First consider $j\in \mathcal{S}$ satisfying $\dist{(x, Q_j^k)} \geq {\sqrt[d]{\kappa_d r_{k}^{d-2}r_{k-1}^2} \over 8}$. 
 In this case we apply Lemma \ref{kernelestimates} with $\nu_1 = \chi_{Q_{j}^{k}}{m_d \over r_{k-1}^2} $ and $\nu_2 = \chi_{\widetilde{B}_{j}^{k}}\mu^m$, $R= r_{k-1}$, $r= r_k$ and $x_0 = x_{Q_j^k}$.  This gives

\begin{equation}\label{eqa2}
    \bigg|\int_{Q_j^k} K(x-\zeta){dm_d(\zeta) \over r_{k-1}^2} - \int_{B_j^k} K(x-\zeta)d\mu^m(\zeta)\bigg| \lesssim \frac{(r_k^{d-2}r_{k-1}^2)^{\alpha/d}}{r_{k-1}^2}
    \int_{Q_j^k}{dm_d(\zeta)\over |x-\zeta|^{d-2+\alpha}} + \int_{B_j^k} {r_k^{\alpha}d\mu^m(\zeta)\over |x-\zeta|^{d-2+\alpha}},
\end{equation}

There can be at most $C$ indices $j\in \mathcal{S}$ satisfying  $\dist{(x, Q_j^k)} \leq {\sqrt[d]{\kappa_d r_k^{d-2}r_{k-1}^2} \over 8}$.  For such a $j$ we employ the elementary fact that for any set $S$ with finite measure
\begin{equation}\label{kernelons}
    \int_{S}|K(\zeta)|dm_d(\zeta)\lesssim \int_S\frac{1}{|\zeta|^{d-2}}dm_d(\zeta) \lesssim \sqrt[d]{(m_d(S))^2}.
\end{equation}
Combined with (\ref{absolutevalue}) this results in
\begin{equation}\label{eqfornotj}
\bigg|\int_{Q_j^k} K(x-\zeta){dm_d(\zeta) \over r_{k-1}^2} - \int_{B_j^k} K(x-\zeta)d\mu^m(\zeta)\bigg| \lesssim \frac{\sqrt[d]{(m_d(Q_j^k))^2}}{r_{k-1}^2} + \frac{\mu^m(B_j^k)}{\epsilon^{d-2}}.
\end{equation}
But now notice that
\begin{equation*}
{\sqrt[d]{\kappa_d r_k^{d-2}r_{k-1}^2} \over 8} \geq \dist{(x, Q_j^k)} \geq  \dist{(x^*, Q_j^k)} - d(x,x^*) \geq {\sqrt[d]{\kappa_d r_k^{d-2}r_{k-1}^2} \over 4} -\epsilon,
\end{equation*}
 and hence we have $\epsilon \gtrsim \sqrt[d]{r_k^{d-2}r_{k-1}^2}$.  Therefore we can bound
$$\frac{\sqrt[d]{(m_d(Q_j^k))^2}}{r_{k-1}^2} + \frac{\mu^m(B_j^k)}{\epsilon^{d-2}}\lesssim \Bigl(\frac{r_k}{r_{k-1}}\Bigl)^{2(d-2)/d}\lesssim \delta_k^2.
$$
Altogether, we can therefore estimate
$$A_3 = \Bigl|\int_{\bigcup_{j\in \mathcal{S}}Q_j^k}K(x-\zeta)\frac{dm_d(\zeta)}{r_{k-1}^2} - \int_{\bigcup_{j\in \mathcal{S}}Q_j^k}K(x-\zeta)d\mu^m(\zeta)\Bigl|,$$
by a sum of two terms:  The contribution to $A_3$ from $j$ which satisfy $\dist{(x, Q_j^k)} \geq {\sqrt[d]{\kappa_d r_{k}^{d-2}r_{k-1}^2} \over 8}$ is at most a constant multiple of
$$\frac{(r_k^{d-2}r_{k-1}^2)^{\alpha/d}}{r_{k-1}^2}
    \int_{B(x, 2r_{k-1}\backslash B(x,{\sqrt[d]{\kappa_d r_k^{d-2}r_{k-1}^2} \over 8})}{dm_d(\zeta)\over |x-\zeta|^{d-2+\alpha}} + \int_{\mathbb{R}^d\backslash B(x,{\sqrt[d]{\kappa_d r_k^{d-2}r_{k-1}^2} \over 8})} {r_k^{\alpha}d\mu^m(\zeta)\over |x-\zeta|^{d-2+\alpha}},$$
which is in turn bounded by a constant multiple of $\bigl(\frac{r_k}{r_{k-1}}\bigl)^{\alpha(d-2)/d}+\bigl(\frac{r_k}{r_{k-1}}\bigl)^{2\alpha/d}\lesssim \delta_k^{2\alpha/d}$.  The contribution from the remaining $j$ is at most a constant multiple of $\delta_k^2$, and so we arrive at
$$A_3 \lesssim \delta_k^{2\alpha/d}.$$

Now write
\begin{equation*}
    \bigg|\int_{B^{k-1}(x^*)\setminus B^{k}(x^*)} K(x-\zeta)d\mu^m(\zeta) \bigg| \leq A_3 + A_4,
\end{equation*}
where
\begin{equation*}
  A_4 = \bigg|\int_{\bigcup_{j\in \mathcal{S}}Q_j^k}K(x-\zeta){dm_d(\zeta)\over r_{k-1}^2}\bigg|.
\end{equation*}
Observe that 
\begin{align*}
   & A_4 \leq \bigg|\int_{\cup_{j\in \mathcal{S}}Q_j^k}K(x-\zeta){dm_d(\zeta)\over r_{k-1}^2}- \int_{B^{k-1}(x^*)}K(x-\zeta){dm_d(\zeta)\over r_{k-1}^2}\bigg| + \bigg| \int_{B^{k-1}(x^*)}K(x-\zeta){dm_d(\zeta)\over r_{k-1}^2}\bigg|\\
   & := A_{5}+A_{6}.
\end{align*}

To estimate the term $A_{5}$ we use the equation \eqref{kernelons}:

\begin{equation}\label{measureofdifference}
    A_{5} \lesssim \frac{1}{r_{k-1}^2}{m_d(\cup_{j\in \mathcal{S}}Q_j^k \triangle B^{k-1}(x^*))}^{2\over d} \lesssim\sqrt[d]{\delta_k^2}. 
\end{equation}

Finally, to estimate the term $A_{6}$, notice that  $x\in (1+{\epsilon\over r_{k-1}^2})B^{k-1}(x^*)$. Using the reflectionless property in Lemma \ref{reflectionless} and the inequality \eqref{kernelons}, we get

\begin{align*}
    A_{6} & = \bigg|\int_{(1+{\epsilon\over r_{k-1}}){B}^{k-1}(x^*)\setminus\tilde{B}^{k-1}(x^*)}K(x-\zeta){dm_d(\zeta)\over r_{k-1}^2}\bigg| \lesssim {1\over r_{k-1}^2} m_d\bigg((1+{\epsilon\over r_{k-1}^2}){B}^{k-1}(x^*)\setminus\tilde{B}^{k-1}(x^*)\bigg)^{2\over d}\\&\lesssim\sqrt[d]{\Bigl(\delta_k + {\epsilon \over r_{k-1}^2}\Bigl)^2}.
\end{align*}

This gives the desired claim \eqref{ineqsecondterm} hence finishes the proof for the proposition.

\end{proof}

\subsection{The potential is unbounded when the condition (\ref{moment}) fails}

Now we will show that the potential associated to $\Omega$ is unbounded if the reflectionless property \eqref{refless} (or equivalently the condition (\ref{moment})) fails. So suppose that there exists some $x_0 \in B(0,1)$ such that 
$\int_{B(0,1)}K(x_0-y) dm_d(y) >0.$  Hence, there exists some $r_0 > 0$ such that

\begin{equation*}
\int_{B(0,1)}K(x-y) dm_d(y) >c_0 > 0, \quad \text{for every } x \in B(x_0, r_0) \subset B(0,1).
\end{equation*}

Recalling the notation of our set, put $G_j^n = x_j^n + r_n B(x_0,r_0)$ for all $n$. So we have $$\int_{B_j^n}K(x-y) {dm_d(y)\over r_n^2} > c_0 >0, \text{ for all } x \in G_j^n.$$

Fix $N\in \mathbb{N}$, fix some $x\in E_{n+1}\cap \bigcup_jG_j^n$. A straightforward modification of the analysis of the previous section leads to

\begin{equation*}
 \int_{B^n(x)\setminus B^{n+1}(x)}K(x-y) d\mu(y)  > c_0 - C\delta_n^{2\alpha\over d},
\end{equation*}
and so, if $x\in \cap_{n=1}^N \bigcup_j G^j_n$, then by summing over all $n = 1,\dots, N$ we get
\begin{equation*}
 \int_{\R^d\setminus B^{n+1}(x)}K(x-y) d\mu(y)  > N c_0 - C\sum_{n=1}^N\delta_n^{2\alpha\over d}.
\end{equation*}
Consequently, the potential associated to $\Omega$ does not belong to in $L^{\infty}(\R^d\backslash \supp(\mu)).$

\section{On the non-existence of the principal value integral}


The result of this section will be rather general and be valid for all odd and even kernels (it will also not depend on the particular co-dimension, but we only state the result for co-dimension $2$).

\begin{thm}
\label{noprin}
Suppose that $\Omega:\mathbb{S}^{d-1}\to \R$ is $\alpha$-H\"{o}lder continuous and does not vanish identically on the sphere.  Then provided the sequence $\frac{r_{n+1}}{r_n}$ converges to zero sufficiently quickly, the potential associated to $\Omega$ fails to exist in the sense of principal value, meaning that
$$\lim_{\epsilon \to 0}\int_{\R^d\backslash B(x,\eps)}\frac{\Omega\bigl(\frac{x-y}{|x-y|}\bigl)}{|x-y |^{d-2}}d\mu(y),$$
fails to exist for $\mu$-almost every $x\in \R^d.$
\end{thm}

The main estimate is the following lemma:

\begin{lem}\label{lowerboundforkernel} 
There exists $c_0>0$ denpending on $\Omega$ such that, provided $n$ is sufficiently large, there is a ball $D_j^n \subset B_j^n$ with $m_d(D_j^n) \geq c_0 r_n^d$ such that if $z\in D_j^n$, then

\begin{equation*}
   |\int_{B(z,tr_n)\setminus B(z,sr_n)}K(x-\zeta)d\mu(\zeta)| \geq c_0.  
\end{equation*}
\end{lem}

Before we exhibit the proof of this lemma, let us see how we get the Proposition \ref{noprin} using this. To observe this, set $F = \{z\in E: z\notin \cup_{j}D_j^n \quad\text{for all but finitely many n}\}$ and then it suffices to show $\mu(F) =0$. Also observe that by denoting by $F_n = \{z\in E: z\notin \cup_{j}D_j^m \quad \textit{for all} \quad m\geq n\} $, we have that $F\subset F_n$, so it will suffice to show that $\mu(F_n) =0$ for all $n.$

To show this, note that there exists $d_0>0$ such that for all sufficiently large $m\geq 0$, at most $(1-d_0)\bigl({r_m\over r_{m+1}}\bigl)^{d-2}$ cubes $Q_l^{m+1}$ fail to intersect $D_j^m$, thus 

\begin{align*}
&\mu\bigg(\bigcup_{l}\bigg\{B_{l}^{m+1}: B_{l}^{m+1}\cap D_j^m = \emptyset\bigg\}\bigg) \leq (1-d_0)\bigg({r_m\over r_{m+1}}\bigg)^{d-2}r_{m+1}^{d-2} \leq (1-d_0)\mu(B_j^m).
\end{align*}
Whenever $n$ is large enough, this inequality can be iterated to get
\begin{equation*}
    \mu(\{z\in E: z\notin D_j^{n+k} \quad \text{for} \quad k = 1, 2,......, m\}) \leq (1 - d_0)^m,
\end{equation*}
which shows that $\mu(F_n) = 0.$

In order to prove the Lemma \ref{lowerboundforkernel}, we use the following result, which relies on an application of Lemma \ref{kernelestimates}.

\begin{clm}\label{claimforkernel}
 For $r,s\in (0,2)$.  For $n \in \mathbb{Z}_{+}$ sufficiently large and any disc $B_j^n$, $z\in \mathbb{R}^d$, we have

\begin{equation*}
    \bigg|\int_{B^n_j\cap [B(z,tr_n)\setminus B(z,sr_n)]}K(x-\zeta)d(\mu - {m_d \over r_n^{2}})(\zeta)\bigg| \lesssim \delta_{n+1}^{\alpha}.
\end{equation*}
\end{clm}

\begin{proof}
To prove this, let us first denote by $A(z, r_n) = B(z,tr_n)\setminus B(z,sr_n)$. Next we suppose that some cube $Q_l^{n+1}\in A(z, r_n)$. Then using Lemma \ref{kernelestimates}, we get 

\begin{equation*}
    \bigg|\int_{Q_l^{n+1}}K(x-\zeta)d(\mu - {m_d \over r_n^{2}})(\zeta)\bigg| \lesssim \Bigl(\frac{r_{n+1}}{r_n}\Bigl)^{d-2+\alpha}+ \Bigl(\frac{r_{n+1}}{r_n}\Bigl)^{(d-2)(1+\alpha/d)}\lesssim \Bigl(\frac{r_{n+1}}{r_n}\Bigl)^{(d-2)(1+\alpha/d)}.
\end{equation*}

If we instead have $Q_l^{n+1} \cap \partial A(z, r_n) \neq \emptyset$, then we have the crude estimate
\begin{equation*}
\bigg|\int_{Q_l^{n+1}\cap A(z, r_n)}K(x-\zeta)d(\mu - {m_d \over r_n^{2}})(\zeta)\bigg| \lesssim \Bigl(\frac{r_{n+1}}{r_n}\Bigl)^{d-2}.
\end{equation*}

There are at most $({r_n\over r_{n+1}})^{d-2}$ squares $Q_l^{n+1}$ contained in $A(z, r_n)$ and observe that no more than $C\bigl({r_n\over r_{n+1}}\bigl)^{(d-2)(d-1)/d}$ cubes $Q_l^{n+1}$ intersect $\partial A(z, r_n)$. 

On the other hand observe that the set $\tilde{A}$ consisting of points in $A(z, r_n) \cap B_j^n$ that are not contained in any cube $Q_l^{n+1}$ has measure that is no greater than $$m_d\Bigl(\Bigl[\bigcup\{Q^{n+1}_{\ell}: Q^{n+1}_{\ell}\subset (1+\delta_{n+1})B_j^n\}\Bigl]\triangle B^n_j\Bigl)\lesssim \delta_{n+1}r_n^d.$$ Whence, $\int_{\tilde{A}}|K(z-\zeta)|{dm_d(\zeta)\over r_n^2} \lesssim\delta_{n+1}.$

Bringing all these estimates together yields
$$ \bigg|\int_{B^n_j\cap [B(z,tr_n)\setminus B(z,sr_n)]}K(x-\zeta)d(\mu - {m_d \over r_n^{2}})(\zeta)\bigg|\lesssim \Bigl(\frac{r_{n+1}}{r_n}\Bigl)^{\alpha(d-2)/d}+ \Bigl(\frac{r_{n+1}}{r_n}\Bigl)^{(d-2)/d}+\delta_{n+1}\lesssim \delta_{n+1}^{\alpha},$$
as required.
\end{proof}

\begin{proof}If $\Omega$ does not vanish identically, then there exists $c>0$, $z_0\in B(0,1)$ and $r,s\in (0,1)$ such that
$$\int_{B(0,1)\cap [B(z_0, r)\backslash B(z_0, s)]}K(z-\zeta)dm_d(\zeta)\geq c.$$
(One can select $z_0$ close enough to $0$ in the correct direction, $r=2$ and $s=1$.)
Now observe next that have $\int_{B(0,1)\cap [B(z,tr_n)\setminus B(z,sr_n)]}K(z-\zeta)dm_d(\zeta)$ is a continuous function. Hence, then there exists some $\alpha >0$, such that if $z\in B(z_0, \alpha) \subset B(0,1)$ we have

\begin{equation*}
  \bigg|\int_{B(0,1)\cap [B(z,t)\setminus B(z,s)]}K(z-\zeta)dm_d(\zeta)\bigg| \geq {c\over 2}.   
\end{equation*}

Therefore, with $D_n^j = B(z_j^n+r_nz, \alpha r_n)$ we have

\begin{equation*}
  \bigg|\int_{B^{n}_j\cap B(z,tr_n)\setminus B(z,sr_n)}K(x-\zeta){dm_d(\zeta)\over r_n^2}\bigg| \geq \frac{c}{2} \text{ for every }z\in D^n_j.   
\end{equation*}

Now applying the Claim \ref{claimforkernel}, we deduce for all $z\in D_j^n$, $|\int_{B(z,tr_n)\setminus B(z,sr_n)}K(z-\zeta)d\mu(\zeta)| \geq c/2 - C\delta_{n+1}^{\alpha}$. Finally, note that the right hand side is atleast ${c\over 4}$ for $n$ sufficiently large, and Lemma \ref{lowerboundforkernel} is proved for a suitable choice of $c_0$.
\end{proof}

\medskip

\vspace{0.3cm}

Benjamin Jaye
School of Mathematics,
Georgia Institute of Technology,
Atlanta, GA USA 30332

\smallskip
{\it E-mail}: \texttt{bjaye3@gatech.edu}
\vspace{0.3cm}

Manasa N. Vempati,
School of Mathematics,
Georgia Institute of Technology,
Atlanta, GA USA 30332

\smallskip
{\it E-mail}: \texttt{nvempati6@gatech.edu}

\end{document}